\documentclass{amsart}

\usepackage{amsmath, amssymb}
\usepackage{verbatim}
\usepackage{graphicx}
\usepackage{latexsym}
\usepackage{amscd}
\usepackage{graphics}
\newcommand{\bd}{\partial}

\newcommand{\RR}{\ensuremath{\mathbb{R}^2}}
\newcommand{\RRR}{\ensuremath{\mathbb{R}^3}}

\newcommand{\irr}{ir\-re\-du\-ci\-ble}
\newcommand{\inc}{in\-com\-press\-i\-ble}

\newcommand{\tm}{3-manifold}
\newcommand{\p}{^{\prime}}

\newcommand{\birr}{$\bd$-\irr}
\newcommand{\binc}{$\bd$-\inc}
\newcommand{\ds}{\displaystyle}

\begin{document}
\title[Concordance of Seifert surfaces]
{Concordance of Seifert surfaces}
\author{Robert Myers} 
\address{Department of Mathematics, Oklahoma State University, Stillwater, 
OK 74078}
\email{myersr@math.okstate.edu}




\renewcommand{\labelenumi}{(\theenumi)}
\newtheorem{thm}{Theorem}[section]

\newtheorem{cor}{Corollary}[section]
\newtheorem{lem}{Lemma}[section]
\newtheorem{prop}{Proposition}[section]

\begin{abstract} 
This paper proves that every 
oriented 
non-disk Seifert surface $F$ for 
an oriented 
knot $K$ in 
$S^3$ is smoothly concordant to a Seifert surface $F\p$ for a hyperbolic knot 
$K\p$ of arbitrarily large volume.  This gives a new and 
simpler proof of the result of Friedl and of Kawauchi that every knot is 
$S$-equivalent to a hyperbolic knot of arbitrarily large volume.  
The construction also gives a new and simpler proof of the result of   
Silver and Whitten and of Kawauchi that for every knot $K$ there is a 
hyperbolic knot $K\p$ of arbitrarily large volume and a map of pairs 
$f:(S^3,K\p)\rightarrow (S^3,K)$ which induces an epimorphism on 
the knot groups. An example is given which shows that knot Floer homology 
is not an invariant of Seifert surface concordance. 
The paper also proves that a set of finite volume hyperbolic 3-manifolds with 
unbounded Haken numbers has unbounded volumes.  
\end{abstract}

\maketitle

\section{Introduction}

In what follows the smooth category will always be assumed. 
This paper concerns two equivalence relations on 
oriented 
knots in $S^3$, 
concordance and $S$-equivalence. Knots $K$ and $K\p$ are \textit{concordant} 
if there is a properly embedded oriented 
annulus $A$ in $S^3\times[0,1]$ with 
$A\cap (S^3\times\{0\})=K$ and $A\cap (S^3\times\{1\})=K\p$ 
such that 
$\partial A=K-K\p$. 
Knots 
$K$ and $K\p$ are \textit{$S$-equivalent} if they have Seifert surfaces $F$ and 
$F\p$ with associated Seifert matrices which are equivalent under integral 
congruence and elementary expansions and contractions \cite{Trotter 73}. \\

Concordant knots need not be $S$-equivalent, e.g.\ the trivial knot 
and a slice knot with non-trivial Alexander polynomial, such as the 
stevedore's knot $6_1$. $S$-equivalent knots need not be concordant; 
Kearton \cite{Kearton 04} has shown that every algebraically slice 
knot is $S$-equivalent to a slice knot, but by Casson and Gordon 
\cite{Casson-Gordon} 
there are algebraically slice knots which are not slice knots. \\

The author \cite{Myers 83} proved that every knot is concordant 
to a hyperbolic knot, 
generalizing the result of Kirby and Lickorish \cite{Kirby-Lickorish 79} 
that every knot 
is concordant to a prime knot. Friedl \cite{Friedl 09} and  
Kawauchi \cite{Kawauchi 89a, Kawauchi 89b, Kawauchi 94} 
have given proofs that every knot is 
$S$-equivalent to a hyperbolic knot, generalizing the result of 
Kearton \cite{Kearton 04} that every knot is $S$-equivalent to a prime knot. \\

Friedl noted, citing work of Kearton \cite{Kearton 75, Kearton 04}, 
Levine \cite{Levine 70, Levine 77}, and Trotter \cite{Trotter 73}, that two knots are $S$-equivalent if and only if they have 
isometric Blanchfield pairings. 
He then noted that by combining two results 
of Kawauchi's imitation theory of knots (Theorem 1.1 of \cite{Kawauchi 89b} 
and Properties I and V of \cite{Kawauchi 89a}) one gets that for every knot 
$K$ there is a hyperbolic knot $K\p$ of arbitrarily large volume and a map 
$f:(S^3,K\p)\rightarrow(S^3,K)$ which induces isomorphisms on every quotient 
of the knot groups by their derived subgroups. Friedl then showed that 
this result implies that the Blanchfield pairings are isometric. 
He also added a note in proof that one can combine the existence 
of such knots and maps with another result of Kawauchi (Theorem 2.2 of 
\cite{Kawauchi 94}) to show $S$-equivalence. \\

It is natural to ask whether for every knot $K$ there is a hyperbolic knot 
$K\p$ of arbitrarily large volume to which $K$ is both $S$-equivalent 
and concordant. It turns out that an affirmative answer is implicit 
in Kawauchi's construction in \cite{Kawauchi 89b}. One can see the 
concordance by looking at Figure 7 of that paper for the time interval 
$0\leq t\leq 1$.\\

Silver and Whitten \cite{Silver-Whitten 04} proved that given any 
triple $(G,\mu,\lambda)$ where $G$ is a knot group and $(\mu,\lambda)$ is a 
meridian-longitude pair for $G$ there are infinitely many triples 
$(\widetilde{G},\widetilde{\mu},\widetilde{\lambda})$ where $\widetilde{G}$ 
is the group of a prime knot and there is an epimorphism 
$\phi:(\widetilde{G},\widetilde{\mu},\widetilde{\lambda})\longrightarrow(G,\mu,\lambda)$. 
In \cite{Silver-Whitten 05} they strengthened this to the $\widetilde{G}$ 
being the groups of hyperbolic knots of arbitrarily large volume. 
In a note in proof they added the comment that Kawauchi 
informed them that many of the results in the paper can be found 
in \cite{Kawauchi 89b} and \cite{Kawauchi 92}.\\

The constructions of Kawauchi and of Silver and Whitten mentioned above are 
rather intricate and the proofs somewhat lengthy. In the present paper the 
author gives a simpler and shorter construction and proof. Recall 
that a \textit{Seifert surface} $F$ for a knot $K$ in $S^3$  is a compact, 
oriented surface with boundary $K$. Both $F$ and $K$ will be assumed to 
be oriented, with $\partial F=K$. A Seifert surface $F\p$ for a 
knot $K\p$ will be said to be \textit{concordant} to $F$ if there is 
an  embedding  $h:F\times [0,1]\rightarrow S^3\times [0,1]$ 
such that $h(F\times\{0\})=h(F\times[0,1])\cap (S^3\times\{0\})=F$,  
$h(F\times\{1\})=h(F\times[0,1])\cap (S^3\times\{1\})=F\p$, and there is an 
orientation of $h(F\times [0,1])$ such that $(S^3\times\{0\})\cap
\bd h(F\times [0,1])=F$ and $(S^3\times\{1\})\cap
\bd h(F\times [0,1])=-F\p$. In this case $K$ and $K\p$ are clearly 
concordant. \\ 

They are also $S$-equivalent, which can be seen as follows.\\
 
Let $N$ be a regular neighborhood of $h(K\times[0,1])$ in $S^3\times[0,1]$. 
Let $P$ be the closure of $h(F\times[0,1])-N$. Let 
$Q=h(F\times[0,1])\cap P$. Finally let $R$ be a regular neighborhood 
of $Q$ in $P$. Then $R$ is homeomorphic to $Q\times[-1,+1]$ and hence 
to $F\times[-1,+1]$. Thus one gets a product structure of the form 
$k:F\times[0,1]\times[-1,1]\rightarrow R$ on $R$. 
By abuse of notation this will be denoted by $R=F\times[0,1]\times[-1,1]$. 
Thus we identify $F$ with $F\times\{0\}\times\{0\}$ and $F\p$ with 
$F\times\{1\}\times\{0\}$. \\

Now choose a collection of oriented simple closed curves $a_i$ on $F$
which represents a basis for $H_1(F)$. Identify $a_i$ with 
$a_i\times\{0\}\times\{0\}$. Let $a_i^+=a_i\times\{0\}\times\{+1\}$; 
this is regarded as $a_i$ pushed off $F$ in the positive normal direction. 
Let $v_{i,j}$ be the linking number of $a_i$ and $a_j^+$. The resulting 
matrix is a Seifert matrix $V$ for $K$. Now let $b_i=a_i\times\{1\}\times\{0\}$ 
and $b_i^+=a_i\times\{1\}\times\{+1\}$. Letting $v\p_{i,j}$ be the linking 
number of $b_i$ and $b_i^+$ one gets a Seifert matrix $V\p$ for $K\p$.
Let $A_i$ be the annulus $a_i\times[0,1]\times\{0\}$. 
Let $A_i^+$ be the annulus $a_i\times[0,1]\times\{+1\}$. 
Then $A_i$ joins $a_i$ to $b_i$ and $A_i^+$ joins $a_i^+$ to $b_i^+$.\\

Recall that if $J$ and $J^+$ are disjoint oriented 
simple closed curves in $S^3$ then they bound properly embedded oriented 
surfaces $G$ and $G^+$ in the 4-ball $B^4$ which can be chosen to 
meet in a finite number of points. The linking number of $J$ and $J^+$ 
is then equal to the algebraic intersection number of $G$ and $G^+$. 
See \cite{Rolfsen}, page 136. \\

Regard $S^3\times\{1\}$ as $\partial B^4$. Choose surfaces $G_i$ and $G_i^+$ 
in $B^4$ with boundaries $b_i$ and $b_i^+$, respectively.  
Let $\widehat{G}_i=A_i\cup G_i$ and  $\widehat{G}_i^+=A_i^+\cup G_i^+$. 
Since $A_i\cap A_j^+=\emptyset$, we have that the intersection number of 
$\widehat{G}_i$ and $\widehat{G}_j^+$ is equal to that of $G_i$ and $G_j^+$. 
It follows that the Seifert matrices $V$ of $F$ and $V\p$ of $F\p$ with 
respect to the given bases are the same. \\

The main result of this paper will now be stated. The notation 
$S^3\backslash\backslash K\p$ means the compact manifold obtained from 
$S^3$ by removing the interior of a regular neighborhood of the knot $K\p$. 
Recall that the \textit{Haken number} \cite{Haken} of a compact 3-manifold 
$M$ is the maximum number of compact, connected, 
properly embedded, \inc, boundary incompressible, 
pairwise non-parallel surfaces in $M$.\\

\begin{thm}Let $F$ be a Seifert surface for a knot $K$ in $S^3$. Assume that 
$F$ is not a disk. Then $F$ is concordant to a Seifert surface $F\p$ for a 
knot $K\p$ such that 
\renewcommand{\labelenumi}{(\alph{enumi})}
\begin{enumerate} 
\item $K\p$ is hyperbolic, 
\item $S^3\backslash\backslash K\p$ has arbitrarily large Haken number, 
\item $S^3-K\p$ has arbitrarily large volume, and 
\item there is a map of pairs $f:(S^3,K\p)\rightarrow (S^3,K)$ 
which induces an epimorphism $f_*:\pi_1(S^3-K\p)\rightarrow \pi_1(S^3-K)$.
\end{enumerate} 
\end{thm}

The paper is organized as follows. Section 2 reviews some basic material. 
Section 3 proves that every non-disk Seifert surface for a knot can be 
put in a certain standard position. Section 4 uses standard position 
to prove (a). Section 5 proves (b). The proof that (b) implies (c) 
follows from a more general result, that a set of finite volume hyperbolic 
3-manifolds with unbounded Haken numbers has unbounded volumes. This fact 
appears to be ``quasi-known'' but the author does not know a precise reference 
in the literature and so gives a simple proof in the Appendix. Section 6 
proves (d). Section 7 considers the question of whether Seifert surface 
concordance implies the invariance of more than just the union of the sets 
of invariants of concordance and $S$-equivalence. 
It shows by example that  
although the Alexander polynomial is an invariant of Seifert surface 
concordance its categorification, knot Floer homology, is not. \\

\section{Preliminaries}

As general references on knot theory and on 3-manifolds see 
\cite{Lickorish} and \cite{Jaco}. 
A compact, connected, orientable 3-manifold $M$ will be called 
\textit{excellent} if it is irreducible, boundary-irreducible, 
anannular, atoroidal, and is not a 3-ball. $M$ will be called 
\textit{Haken} if it contains a two-sided incompressible surface. 
By Thurston's uniformization 
theorem (see e.g. \cite{Morgan}) excellent Haken manifolds are hyperbolic.  \\

The following standard technical result will be used to build more complicated 
hyperbolic 3-manifolds out of simpler pieces. A proof can be found in Section 2 
of \cite{Myers 93}.\\

\begin{lem} [Gluing Lemma]
Let $X$ be a compact, connected 3-manifold. Suppose $F$ is a compact, properly
embedded, two-sided 2-manifold in $X$. It is not assumed that $F$ is connected. Let
$Y$ be the 3-manifold obtained by splitting $X$ along $F$. Denote by $F_1$ and $F_2$ 
the two
copies of $F$ in $\partial Y$ which are identified to obtain $X$.
If each component of $Y$ is excellent, $F_1\cup F_2$ and the closure of 
$Y - (F_1\cup F_2)$ are incompressible in $Y$, and each component of $F_1\cup  F_2$ has 
negative Euler characteristic, then $X$ is excellent.\end{lem}

An $n$-\textit{tangle} is the disjoint union 
$\lambda=\lambda_1\cup\cdots\cup\lambda_n$ of properly embedded arcs in 
a 3-ball $B$. This is sometimes denoted by the pair $(B,\lambda)$. It will 
always be assumed that $n\geq2$. When the specific number $n$ of arcs 
is not at issue or is clear from the context $\lambda$ will just be called a 
$\textit{tangle}$.\\

This paper will assume that $B$ is given a product structure of the form 
$[a,b]\times[c,d]\times[e,f]$ with each component $\lambda_i$ of the tangle 
$(B,\lambda)$ joining a point of $(a,b)\times(c,d)\times\{f\}$ to a point of 
$(a,b)\times(c,d)\times\{e\}$. This is done so that one may compose tangles. 
The product of the tangles $(B,\lambda)$ and $(B,\mu)$ will be obtained 
by setting $(B,\lambda)$ on top of $(B,\mu)$ so that the lower endpoint of 
each $\lambda_i$ equals the upper endpoint of each $\mu_i$. \\

The \textit{exterior} of a submanifold of a 3-manifold is the closure of 
the complement of a regular neighborhood of the submanifold. 
A knot or tangle will be called \textit{excellent} if its exterior is 
excellent. In this case by a slight abuse of language the knot or tangle 
will be called hyperbolic.\\

\section{Standard position for Seifert surfaces}

Let $F$ be a non-disk Seifert surface for a knot $K$. This section defines 
a way of presenting $F$ by a diagram in the plane, standard position. 
This presentation will be used to build concordances. \\

The first step is to define an analogue for certain graphs of a plat 
presentation of a knot.\\ 

Recall the idea of a plat presentation for a knot or link. 
Regard $S^3$ as the union of a $3$-ball $B^+$, a copy $B^0$ of $S^2\times[-1,1]$, and a 
$3$-ball $B^-$ with $B^+\cap B^0=S^2\times\{1\}$ and $B^-\cap B^0=S^2\times\{-1\}$. 
Choose $M$ unknotted, unlinked, properly embedded arcs in each of $B^+$ and $B^-$. 
Arrange them so that the projections $p^{\pm}:B^{\pm}\rightarrow D^{\pm}$ onto 
equatorial disks have no crossings and are unnested. Let $p^0:B^0\rightarrow A$ be projection of 
$B^0$ onto an annulus $A$. Join the endpoints of the arcs in $B^+$  to the endpoints 
of the arcs in $B^-$ by arcs in $B^0$ such that projection on $A$ is given by 
a braid $\beta$. In this paper the braid will be chosen so that it is an element of the 
braid group of the plane, not the braid group of the sphere. 
The knot projections will be drawn in \RR\ with regions $U$, $C$ and $L$ representing 
the projections of regions near the knot in 
$B^+$, $B^0$, and $B^-$, respectively. The map $p:\RRR\rightarrow\RR$ denotes this  
local projection. The vertical coordinate in $\RR$ is denoted by $y$. The braid is represented by a box labelled $\beta$. 
The arcs in $B^+$ and $B^-$ are represented by arcs on the top and bottom of the box, 
respectively.\\

Now suppose that $M\geq 2$. Choose $g$ such that $2\leq 2g\leq M$, 
and let $m=M-2g$. Replace $2g$ adjacent arcs in $L$ with the cone $X$ on 
their boundary points having vertex $v$ in the interior of $L$; require 
that $X$ be disjoint from the remaining arcs in $L$.  This gives a 
$1$-complex in \RRR\ consisting of a wedge of $2g$ circles with possibly 
some additional simple closed curve components. 
Restrict attention to those $\beta$ for which there 
are no such additional components.  
By changing $\beta$ one may obtain an isotopic embedding of $W$ such 
that $X$ is to the left of the arcs in $L$. This will be called a 
\textit{plat presentation of a wedge of $2g$ circles}. See Figure 1.\\

\begin{figure}[h]
\scalebox{0.25}{\includegraphics{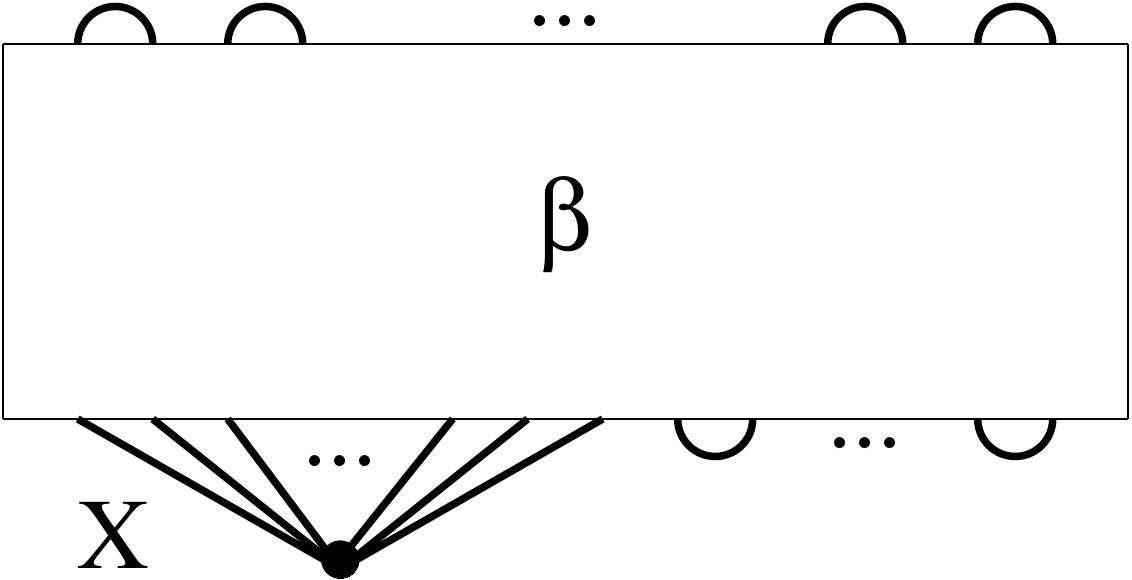}}
\caption{Plat presentation of a wedge of circles}
\end{figure}

One may add words to the top and bottom of $\beta$ to obtain an isotopic 
embedding of $W$ so that the arcs in each of $U$ and $L$ are concentrically 
nested. This will be called a \textit{standard presentation of a wedge 
of $2g$ circles}. See Figure 2.\\

\begin{figure}[h]
\scalebox{0.25}{\includegraphics{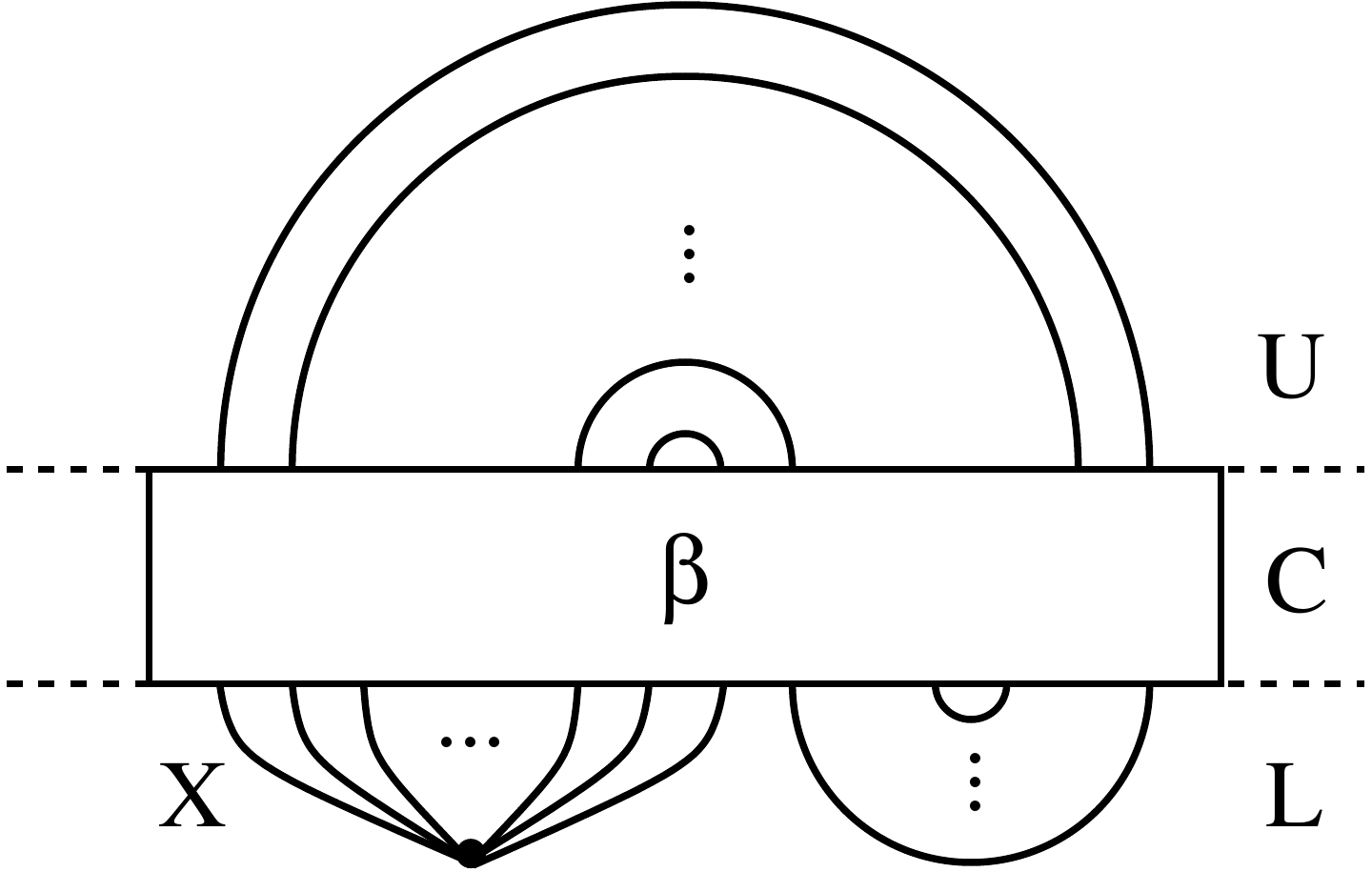}}
\caption{Standard presentation of a wedge of circles}
\end{figure}

Now take regular neighborhoods in $U\cup L$ of the arcs and $X$ and 
widen the arcs in $\beta$ to bands which cross in the same manner as 
the arcs. One gets a \textit{standard presentation of a surface with 
boundary}. The special case of interest is that of an orientable 
surface with connected boundary.\\

\begin{lem} Every non-disk Seifert surface for a knot in $S^3$ has a standard  
presentation.\end{lem}

\begin{proof} Let $F$ be a genus $g\geq1$ Seifert surface for a knot. 
Choose a wedge of circles $W$ in the interior of $F$ such that 
the boundary of a regular neighborhood of $W$ 
in $F$ is parallel in $F$ to $K$. Let $v$ be the vertex of $W$. 
The number of circles is $2g$, where 
$g$ is the genus of $F$. Isotop $F$ so that the projection of 
$W$ onto \RR\ has only tranverse double point singularities. 
In particular $p^{-1}(p(v))$ consists of a single point. 
Isotop $F$ so that $p(v)$ has $y$ coordinate less than or equal to that 
of any other point of $p(W)$. We may assume that for each 
edge of $W$ all the critical values of the function $y\circ p$ are 
local maxima and minima. Isotop $F$ so as to move all the 
minima into $L$  and all the maxima into $U$. 
One now has a plat presentation of $W$. \\

The only obstruction to widening $W$ into a plat 
presentation for $F$ is that some of the bands may 
be twisted. Since $F$ is orientable the twisting in 
each band consists of a number of full twists. 
Each full twist is isotopic to a curl, as illustrated 
in Figure 3. Isotop the local maxima and local minima  
of the curls into $U$ and $L$, respectively. 
The isotopy from plat to standard presentation then 
preserves the fact that the bands are untwisted.
\end{proof}

\begin{figure}[h]
\scalebox{0.125}{\includegraphics{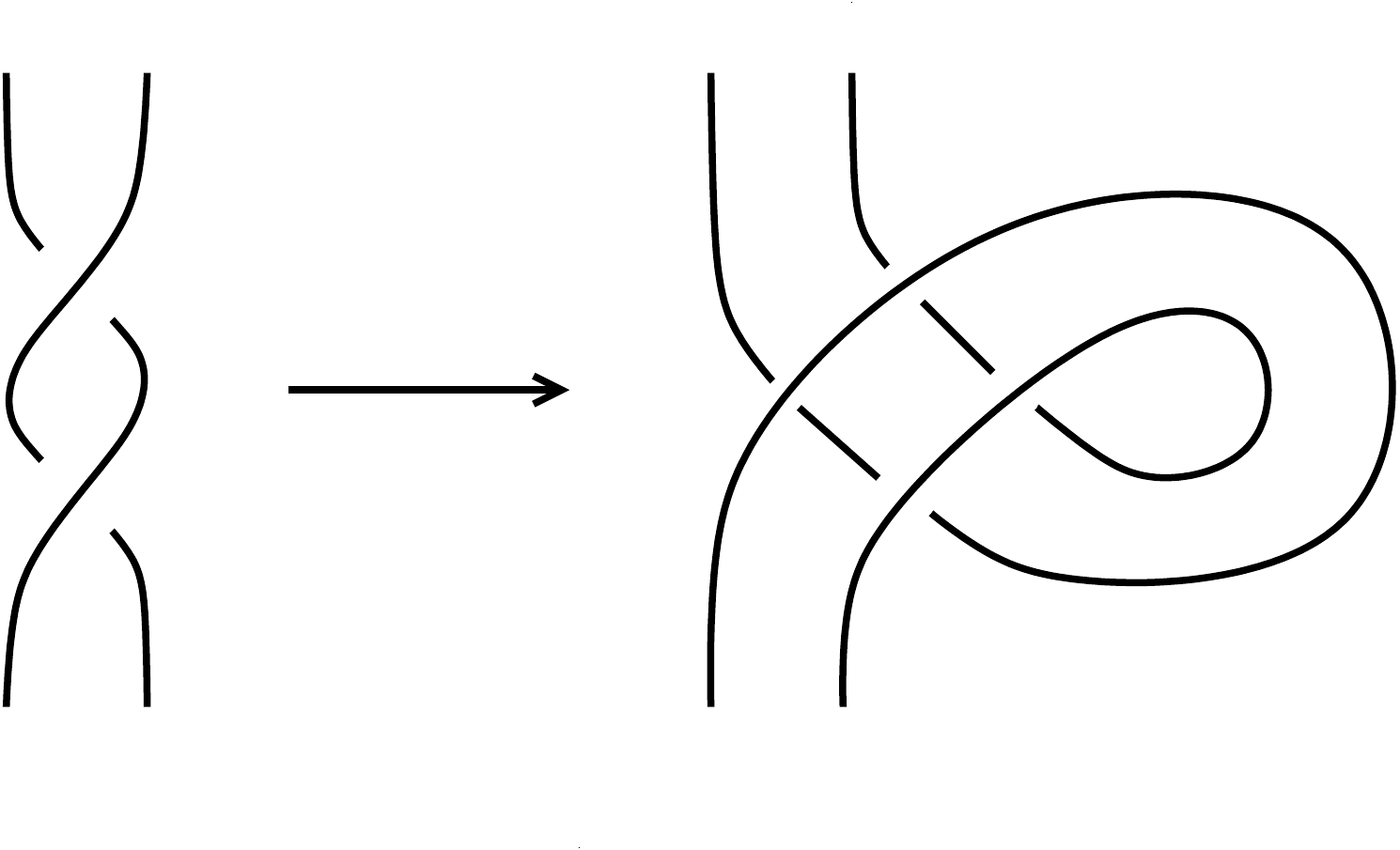}}
\caption{Replacing a twist by a curl}
\end{figure}

\section{Constructing a concordance}

The previous section showed that every non-disk Seifert surface $F$ for a knot 
in $S^3$ has a standard presentation which is obtained by widening a standard 
presentation of a wedge of circles $W$. 
The basic idea for constructing a concordance from $F$ to a surface $F\p$ is to 
construct a concordance from $W$ to some $W\p$ and widen $W\p$ to get $F\p$. 
However, the obvious surface one gets from a projection of $W\p$ might not be 
concordant to $F$. As an example of this phenomenon let $P$ be a zero 
crossing projection of the 
trivial knot and $P\p$ a one crossing projection of the trivial knot. 
Let $A$ and $A\p$ be annuli 
obtained by widening $P$ and $P\p$ into bands. These annuli cannot be concordant 
because the 
components of their boundaries have different linking numbers. 
The way to fix this is to require that any isotopies of a diagram be 
\textit{regular}, i.e. they are composed of Reidemeister moves of types II and III 
together with planar isotopies.

\begin{lem} Given a standard presentation of a non-disk Seifert surface $F$, 
there is a concordance of $F$ with a 
Seifert surface $F\p$ such that $S^3-K\p$ is hyperbolic, where $K\p=\partial F\p$.\end{lem}

\begin{proof} By  the gluing lemma it will be sufficient to show that 
$S^3-F\p$ is excellent.\\

Let $W$ be the wedge of circles of which $F$ is a regular neighborhood.  
Choose a 3-ball $E^+$ in $U$ and a 3-ball $E^-$ in $L$ as shown in 
Figure 4. $E^+$ meets $C$ in a disk and meets $W$ in $M$ straight arcs; these 
arcs meet $C$ in the leftmost $M$ points of $W\cap U\cap C$. 
$E^-$ meets $C$ in a disk and meets $W$ in all but the first endpoint of $X$ and in the 
leftmost  $M-2g$ endpoints of the arcs in $(W\cap L)-X$.

\begin{figure}[h]
\scalebox{0.25}{\includegraphics{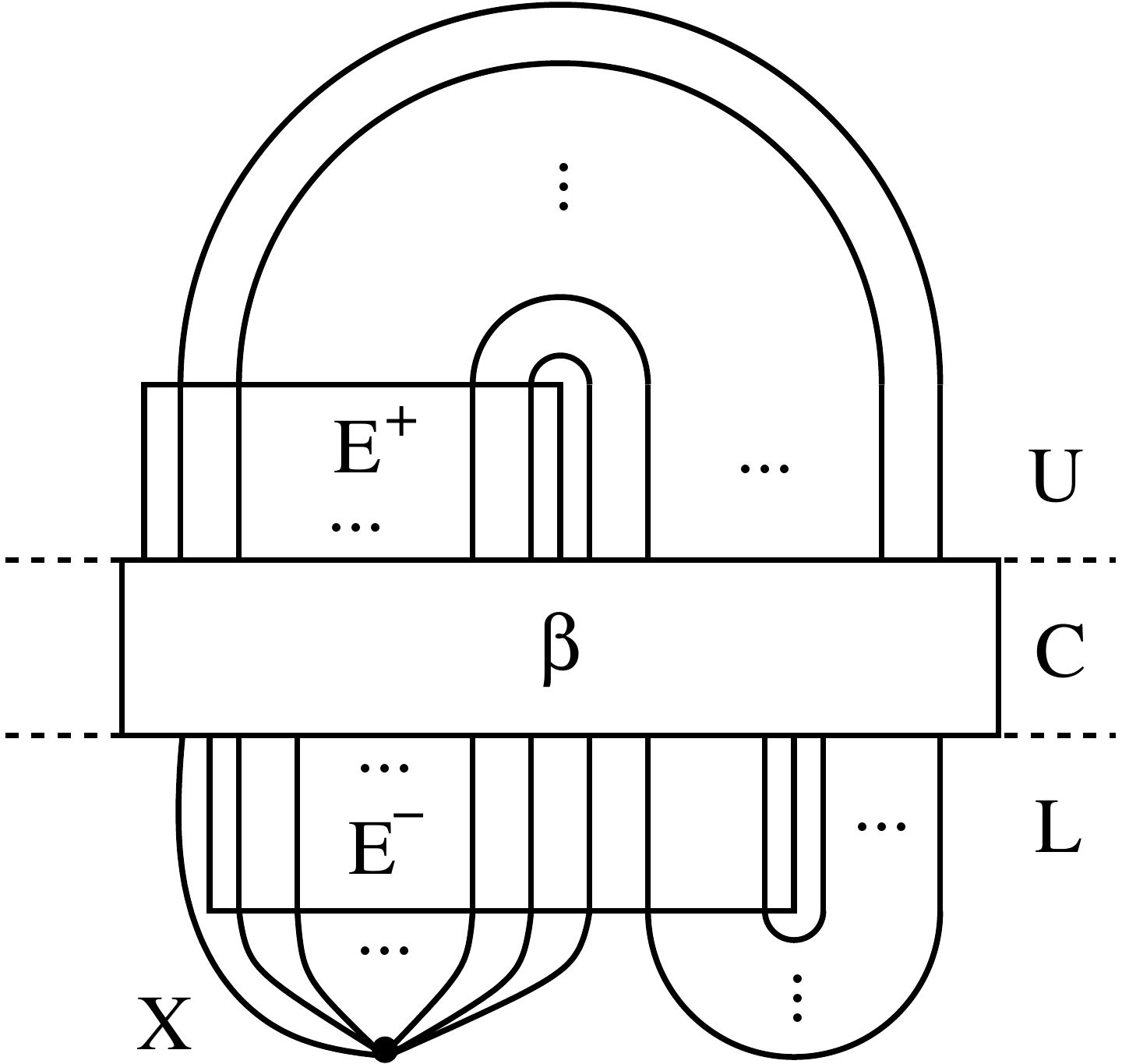}}
\caption{$W$ before surgery}
\end{figure}

The trivial tangles in $E^+$ and in $E^-$ are then replaced by concordant hyperbolic tangles 
$\alpha$ and $\gamma$, respectively. 
The process starts with a hyperbolic $n$-tangle constructed in 
\cite{Myers 83}. Figure 5 shows the $n=4$ case. \\

\begin{figure}[h]
\scalebox{0.25}{\includegraphics{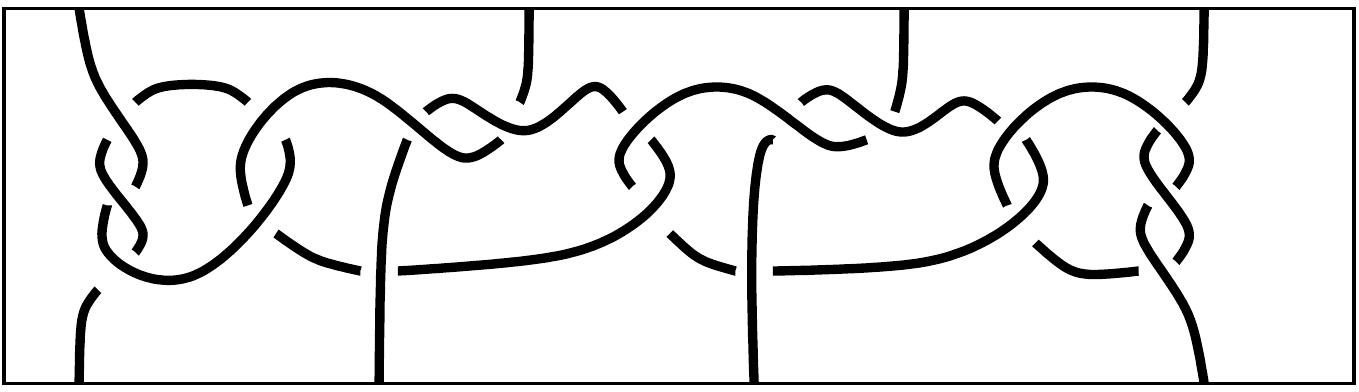}}
\caption{A hyperbolic $n$-tangle, $n=4$}
\end{figure}

This tangle is then composed with its mirror image as in Figure 6. 
By the gluing lemma the new tangle has an excellent Haken 
exterior and is thus hyperbolic.\\

\begin{figure}[h]
\scalebox{0.25}{\includegraphics{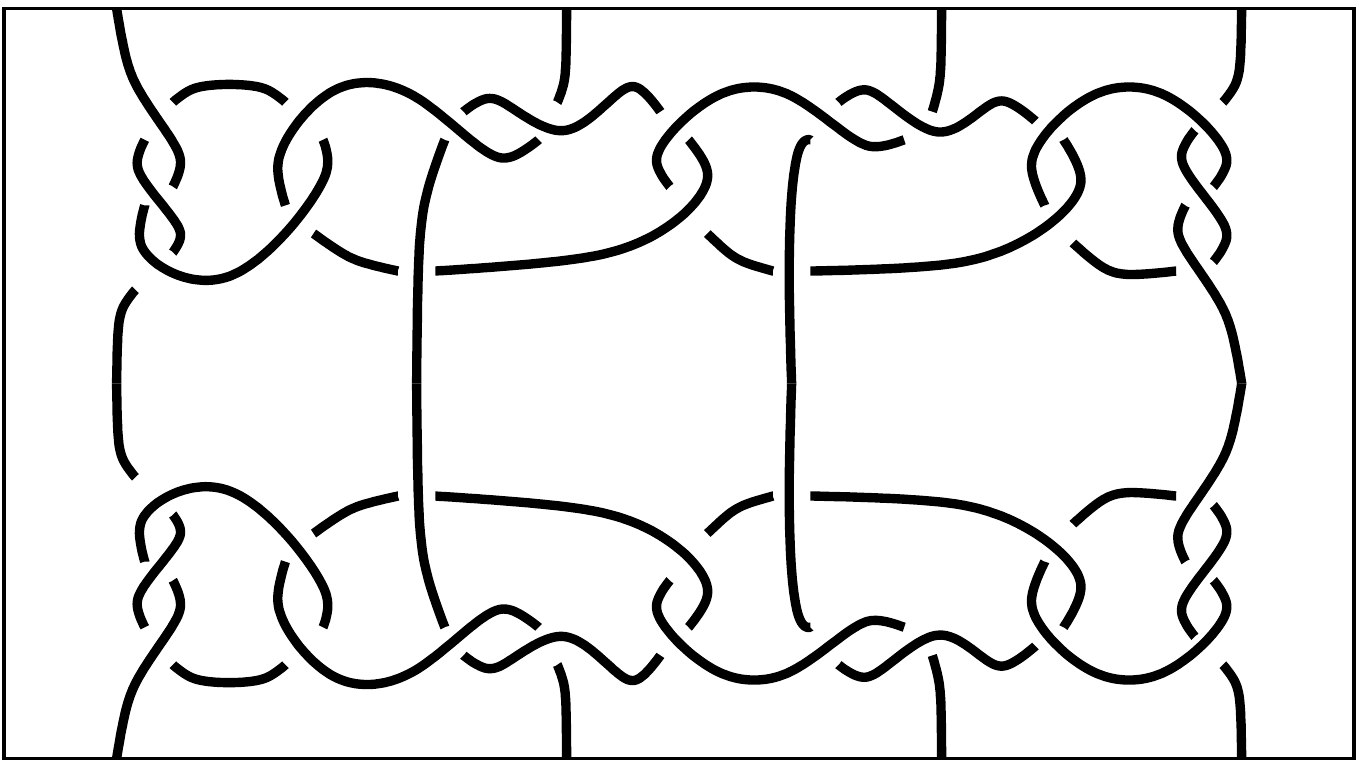}}
\caption{A ribbon $n$-tangle, $n=4$}
\end{figure}

A concordance to the trivial tangle is then constructed by attaching 
$1$-handles, 
performing Type II Reidemeister moves, and 
then attaching $2$-handles. See Figure 7 for the first stage.\\

\begin{figure}[h]
\scalebox{0.25}{\includegraphics{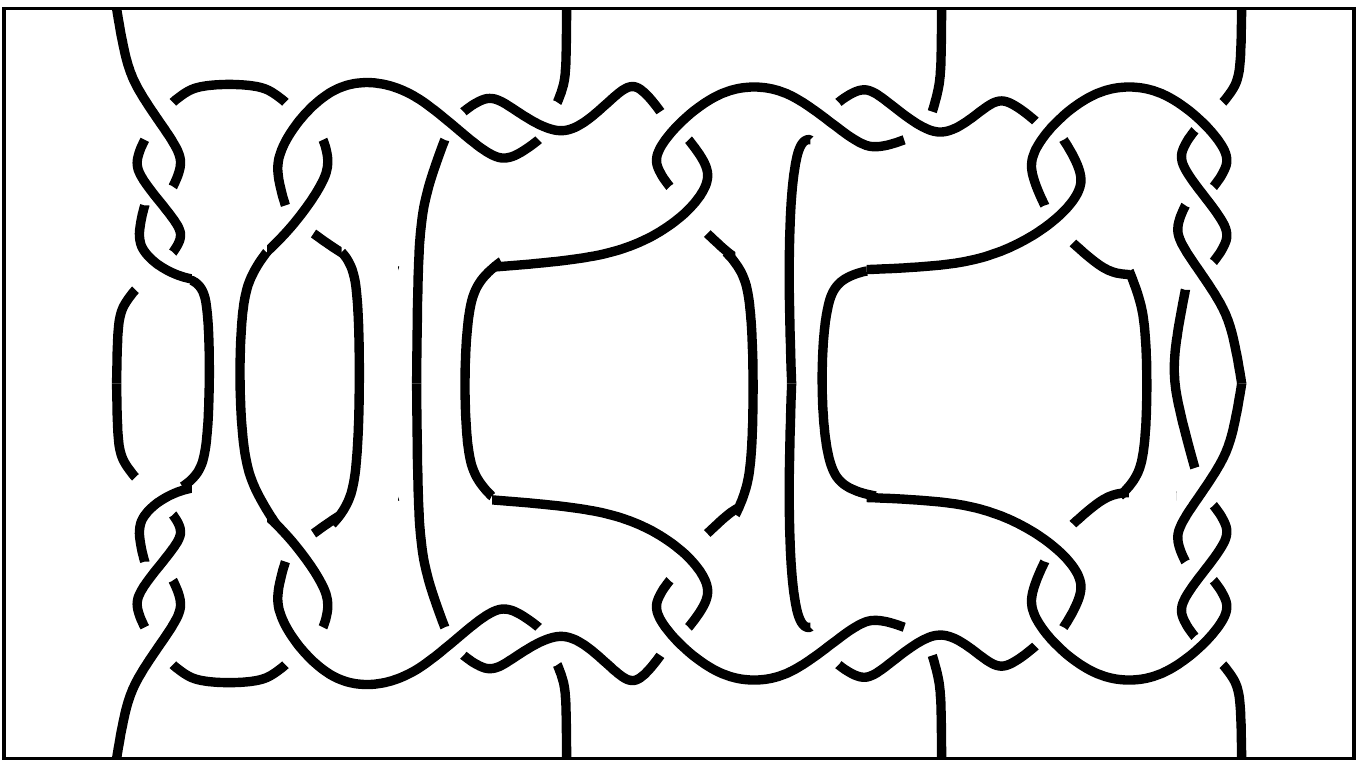}}
\caption{A ribbon concordance, $n=4$}
\end{figure}

For future reference the general version of this result is stated below.\\

\begin{lem} Let $D$ be a disk. 
Let $\tau=\tau_1\cup\cdots\cup\tau_n$ be an $n$-tangle in 
$D\times[0,1]$ such that each $\tau_i$ has one boundary point in 
$int(D)\times\{0\}$ and the other in $int(D)\times\{1\}$. Let 
$\delta(\tau)$ be the tangle in $D\times[-1,1]$ obtained by 
taking the union of $\tau$ and its mirror image in $D\times[-1,0]$. 
Then $\delta(\tau)$ is concordant to the trivial tangle $\varepsilon$.
\hspace{0.75in}$\Box$\end{lem}

Continuing with the proof of Lemma 4.1 
one next constructs the Seifert surface $F\p$ be replacing the disjoint 
union of untwisted bands $F\cap(E^+)$ by the disjoint union of 
untwisted bands in $E^+$ whose centerlines form the tangle $\alpha$. 
These bands are chosen so that their intersections with $\partial E^+$ 
are the same as those of $F$ with $\partial E^+$. A similar construction 
in $\partial E^-$ then completes the construction of $F\p$. \\

One now shows that $S^3-F\p$ is hyperbolic. 
Figure 8 shows the new wedge of circles $W\p$. The exterior of $W\p\cap U$ 
is homeomorphic to the exterior of the tangle $\alpha$ and is 
therefore hyperbolic. The exterior of $W\p\cap L$ is homeomorphic 
to the exterior of the tangle $\gamma$. This can be seen as follows. 
Denote the arcs in $X$ by $\gamma_1$, $\ldots$ $\gamma_k$, numbered 
from left to right. Slide the lower endpoint of $\gamma_2$ along $\gamma_1$ 
and into $C\cap L$. Continue with $\gamma_3$ through $\gamma_{2g-1}$. 
The complement in $L$ of the new set of arcs is homeomorphic to the 
complement of $\gamma$ and is therefore hyperbolic. Since the exterior of 
$\beta$ in $C$ is a product the result follows.\end{proof}

\begin{figure}[h]
\scalebox{0.25}{\includegraphics{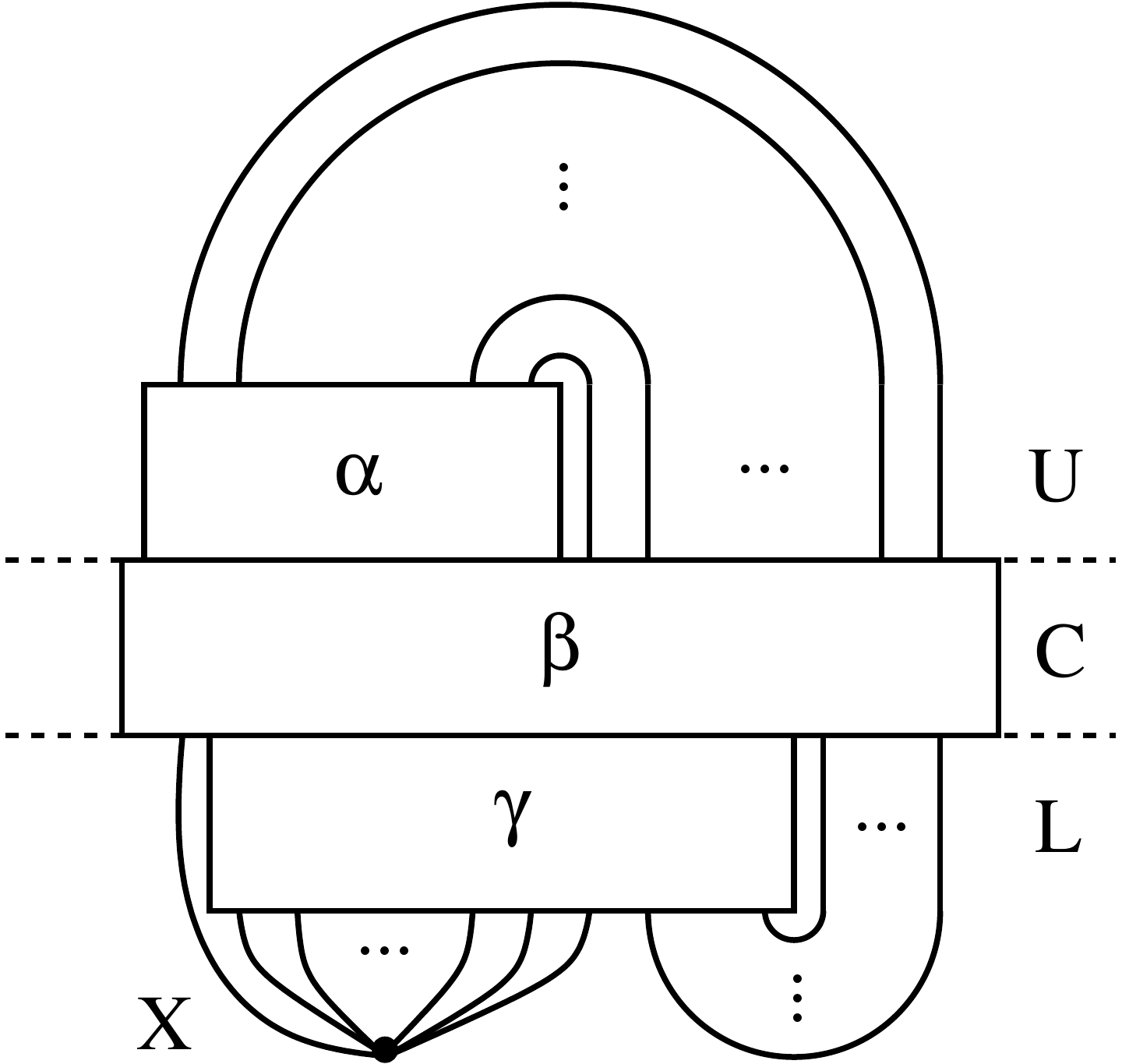}}
\caption{$W\p$ after surgery.}
\end{figure}

\section{Raising the Haken number and the volume}

\begin{lem} Given a standard presentation of a non-disk Seifert surface $F$, 
and given an $N>0$ there is a concordance of $F$ with a 
Seifert surface $F\p$ such that $K\p$ is hyperbolic and $S^3-K\p$ 
has Haken number at least $N$. \end{lem}

\begin{proof} In the construction of the previous section replace each 
of the tangles $\alpha$ and $\gamma$ by $N$ copies of itself stacked 
one on top of the other. This gives $N$ disjoint copies of the exteriors 
of $\alpha$ and $\gamma$. The boundaries of the exteriors of these tangles 
are all incompressible and non-parallel in $S^3-K\p$. By the gluing lemma 
the exterior of the new knot is an excellent Haken manifold and it follows 
that $K\p$ is hyperbolic.\end{proof}

\begin{lem}  Given a standard presentation of a non-disk Seifert surface $F$, 
and given a $V>0$ there is a concordance of $F$ with a 
Seifert surface $F\p$ such that $K\p$ is hyperbolic and $S^3-K\p$ 
has volume at least $N$. \end{lem}

\begin{proof} This follows from Corollary A.1 in the Appendix and Lemma 5.1.\end{proof}

\section{Making a map}

\begin{lem} Given a standard presentation of a non-disk Seifert surface $F$, 
there is a concordance of $F$ with a Seifert surface $F\p$ 
satisfying $\mathrm{(a)}$, $\mathrm{(b)}$, $\mathrm{(c)}$, and $\mathrm{(d)}$. \end{lem}

\begin{proof} The proof follows from Lemmas 6.2 and 6.3 below by defining 
$f$ to be the identity outside the tangles involved.\end{proof}

A tangle $(B^3,\tau_1\cup\ldots\cup\tau_n)$ is a \textit{boundary tangle} 
if there are disjoint arcs $\sigma_1,\ldots,\sigma_n$ in $\partial B^3$ 
such that $\partial\sigma_i=\partial\tau_i$ and disjoint, compact, orientable surfaces 
$G_i$ in $B^3$ such that $\partial G_i=\tau_i\cup\sigma_i$. 
Let $\tau=\tau_1\cup\ldots\cup\tau_n$ and $G=G_1\cup\ldots\cup G_n$.\\
 
Let $\tau^*=\tau^*_1\cup\ldots\tau^*_n$ be a trivial tangle in 
$B^3$ with $\partial \tau^*_i=\partial \tau_i$ for all $i$. 
Choose disjoint disks  
$G_i^*$ in $B^3$ with $\partial G_i^*=\partial G_i$.\\

\begin{lem} There is a map 
$g:(B^3,\tau,B^3-\tau)\rightarrow(B^3,\tau^*,B^3-\tau^*)$ 
which is the identity on $\partial B^3$ and a homeomorphism from 
$\tau$ to $\tau^*$. In particular $g$ induces an epimorphism 
$\pi_1(B^3-\tau)\rightarrow\pi_1(B^3-\tau^*)$ which carries the meridians  
of  $\tau$ to the meridians of $\tau^*$.\end{lem}

\begin{proof} Let $Y$ be the exterior of $\tau$ in $B^3$. Let $H_i=G_i\cap Y$. 
Let $N_i=H_i\times[-1,1]$ be a regular neighborhood of $H_i$ in $Y$. 
Let $N$ be the union of the $N_i$. 
Let $C$ be a collar on $\partial Y$ whose intersection with each $N_i$ 
has the form $A_i\times [-1,1]$, where $A_i$ is a collar on $\partial H_i$ 
in $H_i$. \\

In a similar fashion let $Y^*$ be the exterior of $\tau^*$ in $B^3$, let 
$H^*_i=G_i^*\cap Y^*$, $N_i^*=H_i^*\times [-1,1]$, $N^*=\cup N_i^*$, $C^*=C$, and $A_i^*=A_i$.\\

For each $t\in[-1,1]$ define a map from $N_i$ to $N_i^*$ by crushing 
$(H_i-C)\times\{t\}$ to a point in $H^*_i$. This defines the restriction 
of $g$ to $H_i\times[-1,1]$. \\

One next defines $g$ on the closure $W$ of the complement of $N$ in $B^3$  
by crushing $W-C$ to a point. This gives a quotient map onto a $3$-ball 
which may be identified with the closure of $B^3-N^*$. \\

Putting the two quotient maps together gives a quotient map from 
$Y$ to $Y^*$ which extends to a map $g:B^3\rightarrow B^3$ with the 
required properties.\end{proof}

The existence of hyperbolic boundary $n$-tangles was proven by 
Cochran and Orr \cite[Lemma 7.3 on pp. 519-520]{Cochran-Orr} in a more abstract general setting. In keeping with 
the desire to make the constructions in this paper as explicit 
as possible their procedure is implemented in the following specific 
construction.\\

\begin{lem}[Cochran-Orr]  Hyperbolic boundary $n$-tangles exist.\end{lem}

\begin{proof} First choose a hyperbolic $2n$-tangle $\lambda$. Configure it 
so that the ambient 3-cell is a rectangular box with each component of 
the tangle joining the interior of the top of the box to the interior of 
the bottom of the box. Place it in the interior of a larger 
box with which it is concentric. Connect the endpoints of $\lambda$ to the 
boundary of the larger box by straight arcs  
as in the first diagram in Figure 9. Then slide the endpoints of every second 
arc onto the arc preceding it as in the second diagram. 
Then slide the bottom endpoints of each resulting graph across the 
front of the larger box onto the top 
arc as in the third diagram. This does not change the homeomorphism type 
of the exterior of the graph. \\

\begin{figure}[h!]
\scalebox{0.125}{\includegraphics{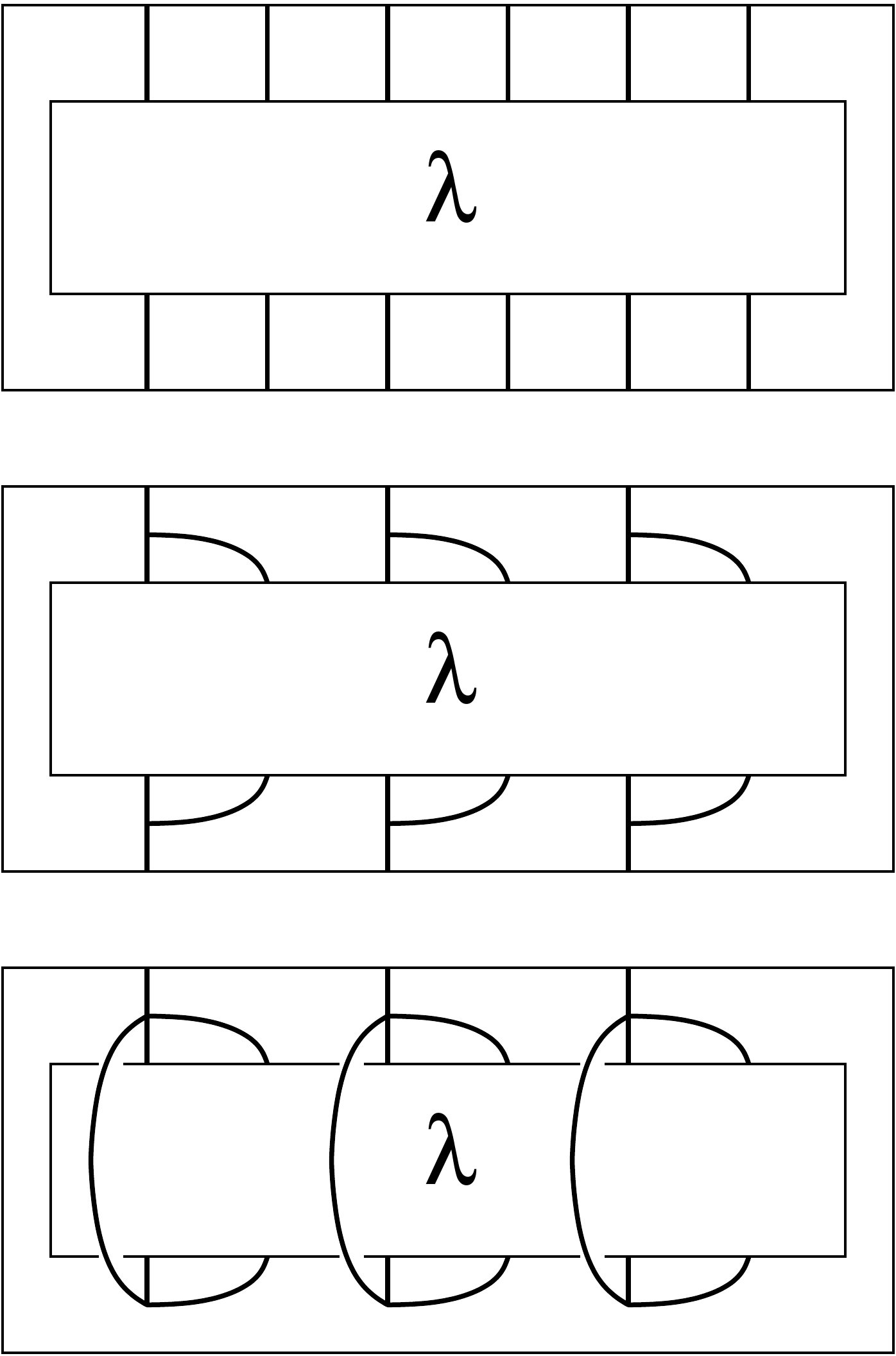}}
\caption{Sliding endpoints to obtain a graph}
\end{figure}

Then one constructs the $2n$-tangle $\Delta(\lambda)$ by 
replacing each arc of $\lambda$ by 
two parallel copies as in Figure 10.\\

\begin{figure}[h]
\scalebox{0.25}{\includegraphics{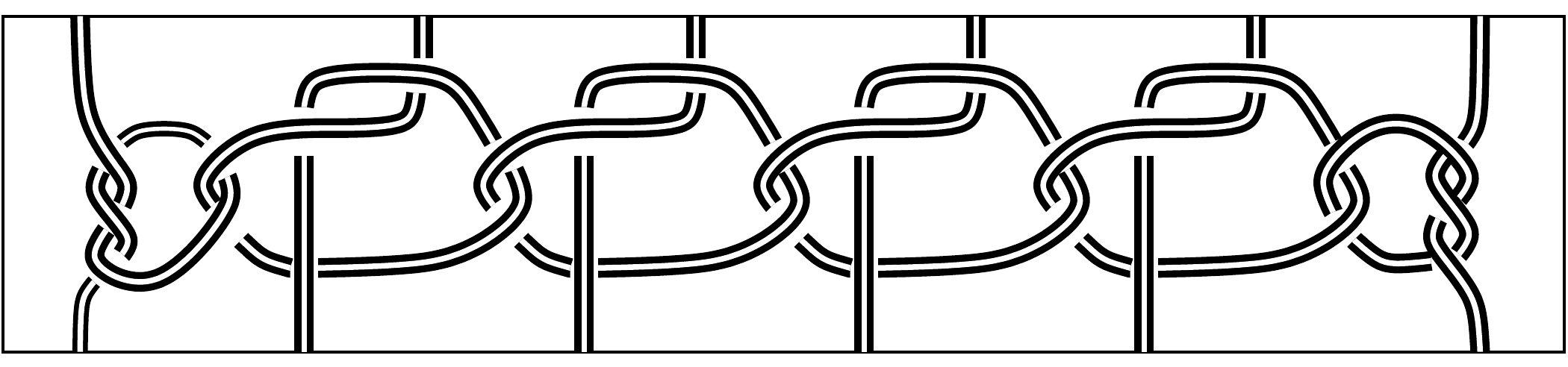}}
\caption{Double $\Delta(\lambda_n)$ of $\lambda_n$, $n=6$}
\end{figure}

Next one modifies the last diagram of Figure 9 in the following 
ways to obtain Figure 11. First, one replaces $\lambda$ in the inner box by 
$\Delta(\lambda)$. 
Second, one widens the graphs which join the inner box to the boundary of 
the outer box to obtain surfaces whose unions with the bands inside the 
inner box are punctured tori.    
Note the half-twists inserted into the lower portion of Figure 11 
to achieve this. By the gluing lemma the exterior of this new tangle 
is hyperbolic since it is obtained from a $3$-manifold homeomorphic to 
the exterior of $\lambda$ by identifying pairs of incompressible 
once-punctured tori in its boundary.\\

\begin{figure}[h]
\scalebox{0.25}{\includegraphics{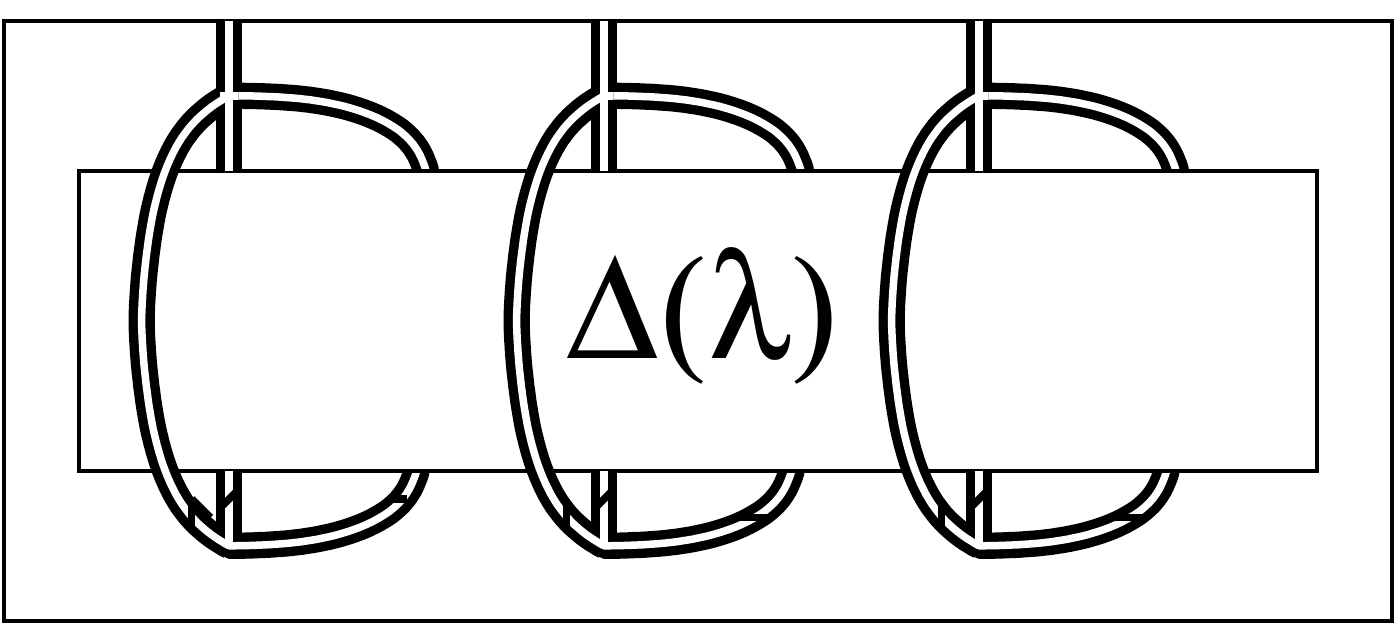}}
\caption{A hyperbolic boundary tangle}
\end{figure}

Finally one slides the left endpoints of each arc across the front of 
the larger box (dotted lines) to obtain the final hyperbolic boundary 
tangle as in 
Figure 12. The union of this tangle with its mirror image is still 
a boundary tangle and the rest of the proof proceeds as before.
\end{proof}

\begin{figure}[h]
\scalebox{0.25}{\includegraphics{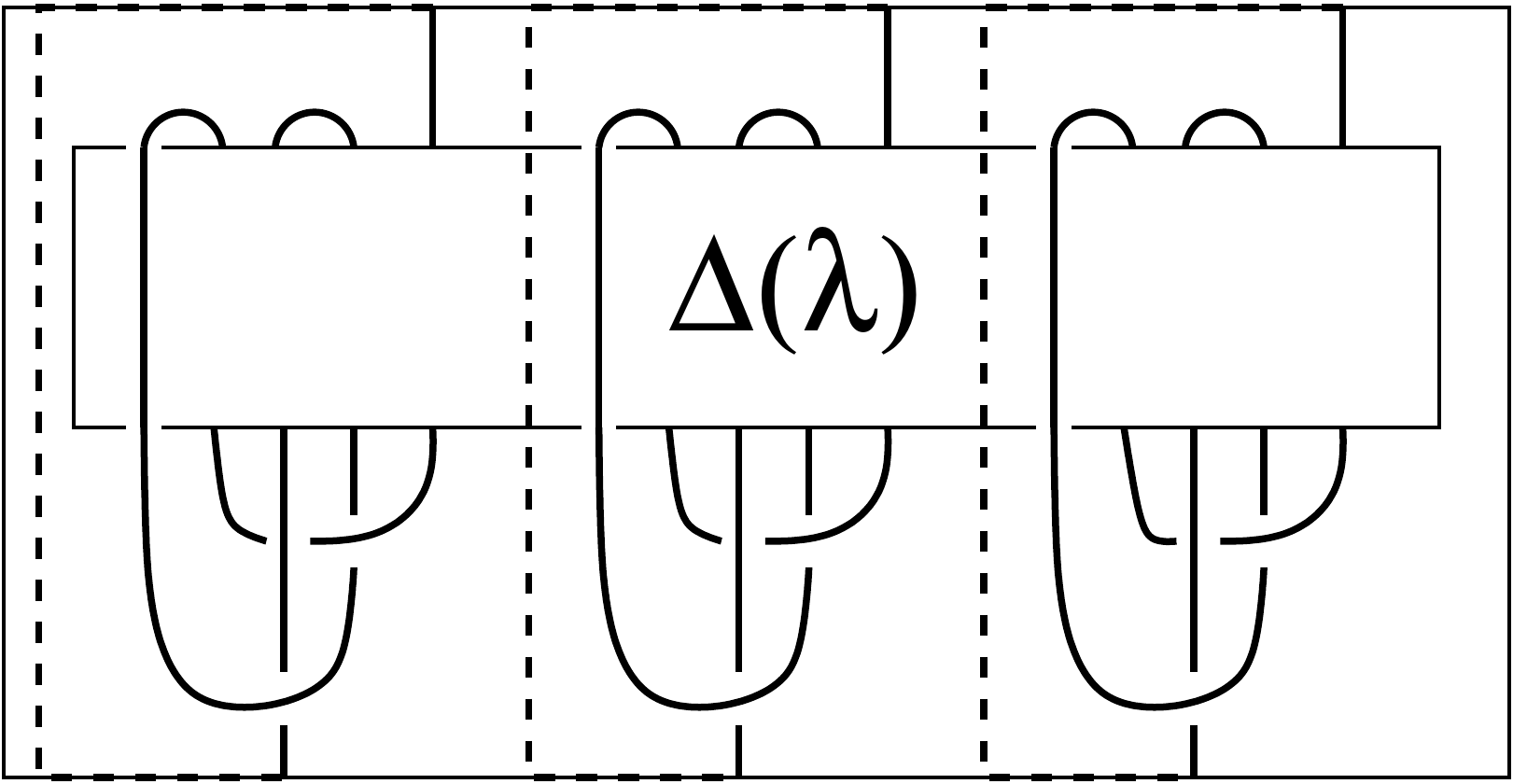}}
\caption{The hyperbolic boundary tangle ready for use}
\end{figure}

\section{Non-invariance of knot Floer homology}

This section gives an example of a knot $J$ with Seifert surface $F$ and a knot $J\p$ with 
Seifert surface $F\p$ such that $F$ and $F\p$ are concordant, but $J$ and $J\p$ have different knot Floer homology.\\

$J$ is the trefoil knot and $F$ a genus one Seifert surface. $J\p$ is a certain twisted 
double of a copy $K$ of the stevedore's knot $6_1$ and $F\p$ is a genus one surface contained in a solid torus $V$ 
whose core is $K$.\\

Figure 13 shows two projections of $K$. The second will be used since it 
more clearly displays the fact that $K$ is a ribbon knot. See \cite{Mizuma} for an isotopy between them.\\

\begin{figure}[h]
\raisebox{0.2in}{\scalebox{0.25}{\includegraphics{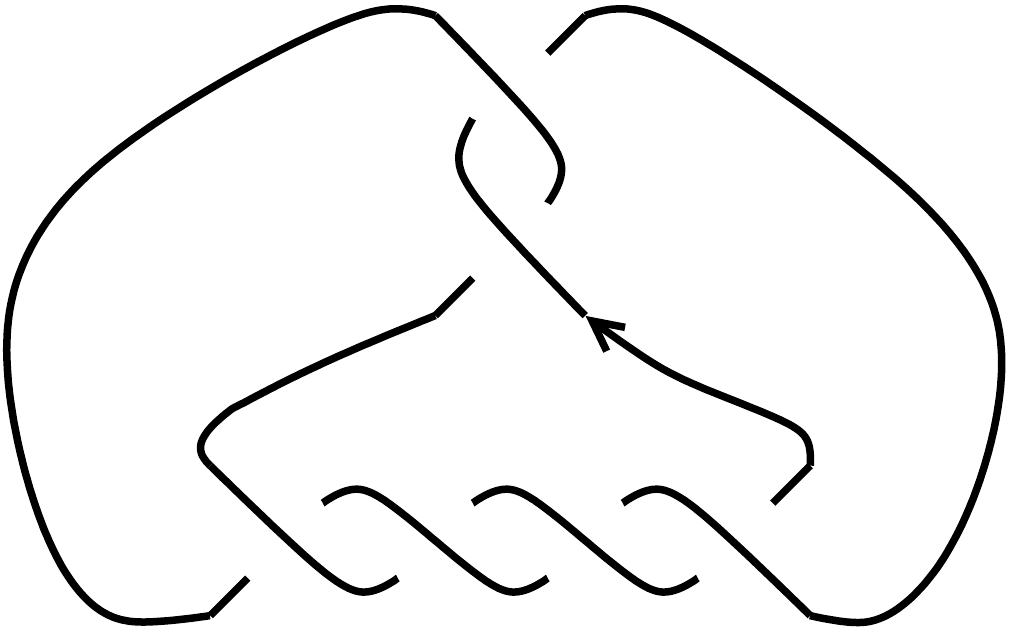}}}
\hspace{0.5in}
\scalebox{0.20}{\includegraphics{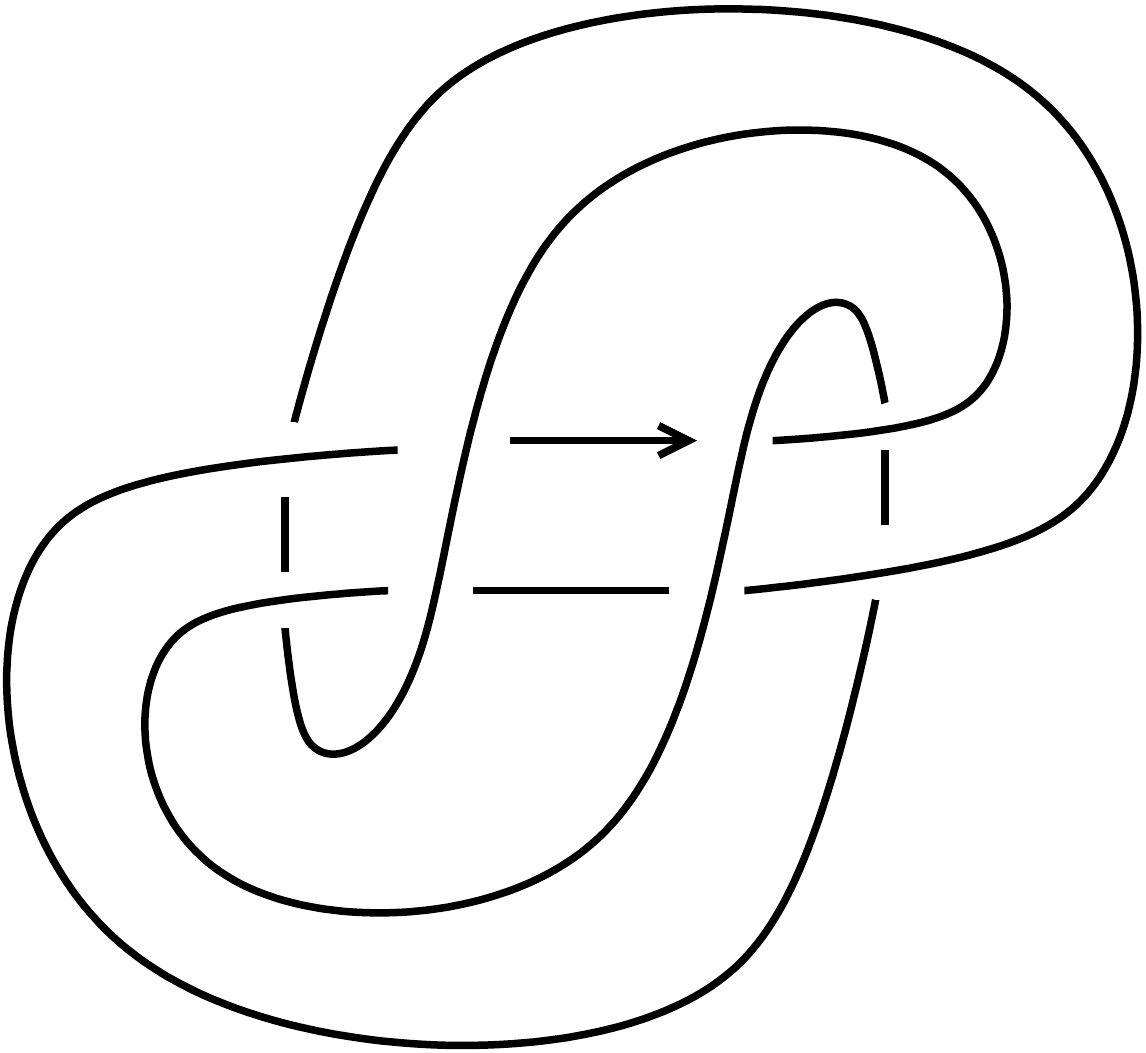}}
\caption{Two projections of the stevedore's knot $6_1$.}
\end{figure}
 
Figure 14 shows an annulus $A$ embedded in 
$S^3$ with one boundary component being $K$. The orientations 
of $K$ and the other boundary component $K^*$ are chosen so that 
the two curves are homologous in $A$. From the diagram one computes that 
the linking number $\mathrm{lk}(K,K^*)=0$. Let $\widetilde{K}=K\cup K^*$.\\

\begin{figure}[h]
\scalebox{0.2}{\includegraphics{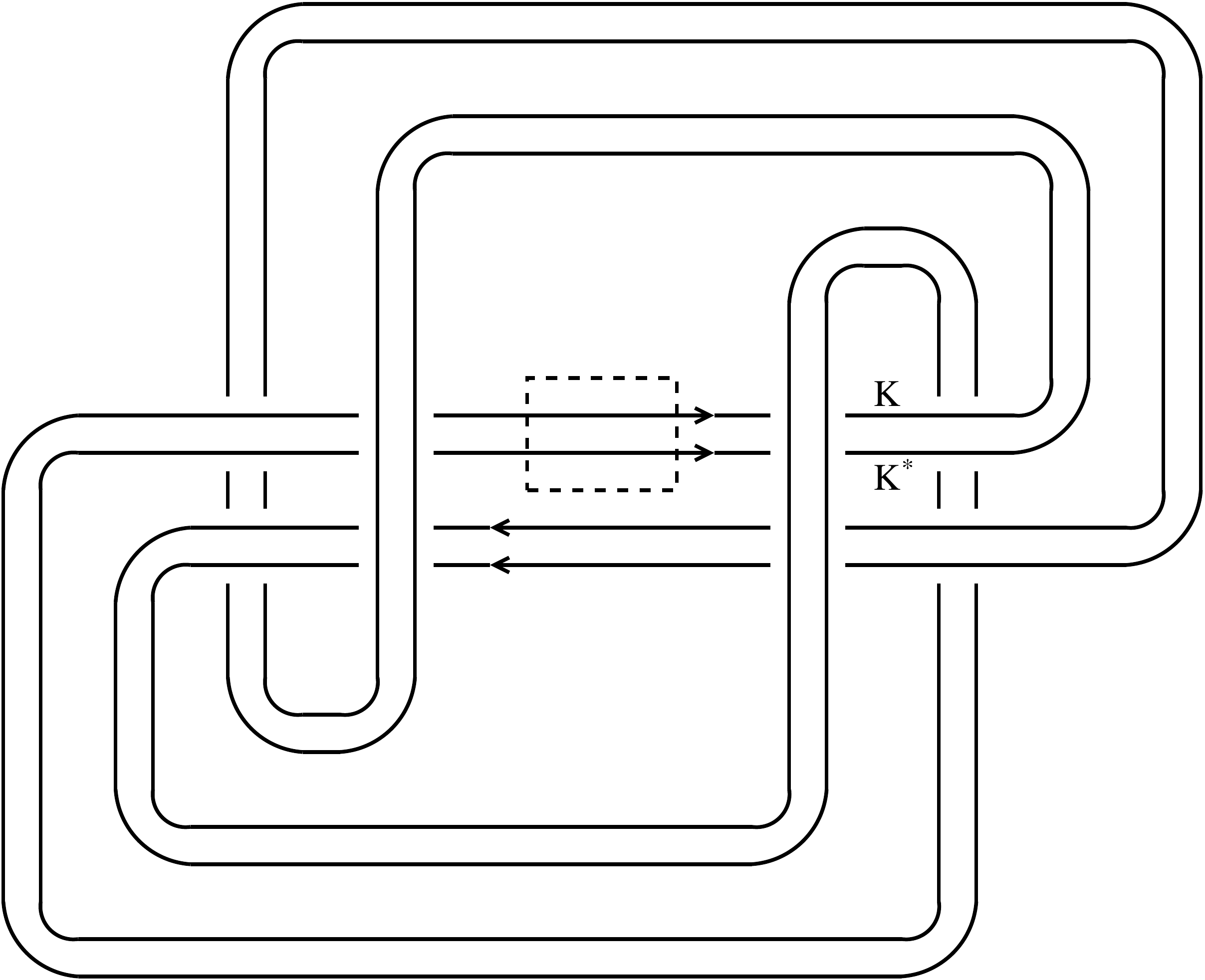}}
\caption{A 2-strand cable link $\widetilde{K}=K\cup K^*$ of $K=6_1$.}
\end{figure}

The knot $J\p$ is obtained by reversing the orientation on  
$K^*$ to get a new oriented link $\widehat{K}=K\cup(-K^*)$ and then  
replacing the trivial tangle in the box in 
Figure 14 by the tangle in Figure 15. \\

\begin{figure}[h!]
\scalebox{0.25}{\includegraphics{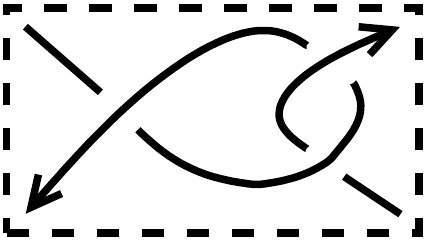}}
\caption{}
\end{figure}

This results in the annulus being replaced by a genus one surface $F\p$. 
A saddle move on the new diagram followed by a pair of 2-handle additions shows that $F$ and $F\p$ are concordant.\\

By a result of Ni \cite{Ni} knot Floer homology detects fibered knots. $J$ is fibered. If $J\p$ were fibered then its 
companion $K$ would also be fibered (Proposition 9.11 of \cite{Myers 80}), but the Alexander polynomial $2t^{-1}-5+2t$ of $K$ 
is not monic, and so $K$ is not 
fibered (see e.g. \cite[Corollary 10.8]{Rolfsen}), thus $J$ is not fibered, and so the knot Floer homologies of $J$ and $J\p$ 
must be different.\\

\appendix

\section{Pumping up the volume} 

There are several results in the literature to the effect that a 
topologically complicated hyperbolic 3-manifold has high volume. 
See for example \cite{Lackenby, Purcell, Shalen}. \\ 

One measure of the complexity of a compact 3-manifold $M$ is the 
\textit{Haken number} $h(M)$ \cite{Haken, Jaco}, the maximum number of 
compact, connected, 
properly embedded, \inc, boundary incompressible, pairwise non-parallel surfaces 
in $M$. Call the union of such a maximal collection of surfaces a 
\textit{Haken system} for $M$.\\

\begin{thm} Let $M_n$ be a sequence of compact, connected, orientable 
3-manifolds with complete, finite volume, hyperbolic interiors $N_n$. \\ 
\begin{center}
If ${\ds\lim_{n\rightarrow\infty}h(M_n)=\infty}$, then 
${\ds\lim_{n\rightarrow\infty}V\!ol(N_n)=\infty}$. 
\end{center}
\end{thm}

\begin{proof} If not, then by passing to a subsequence we may assume that 
$V\!ol(M_n)$ is bounded above by a positive constant $V$. It then follows 
from J\o rgensen's Theorem \cite[Theorem 5.12.1, p.\ 119]{Thurston} 
that there is a finite set $\{X_1,\ldots,X_r\}$ of compact, 
connected, orientable 3-manifolds with finite volume, hyperbolic interiors 
such that for each $M_n$ there is an $X_j$ such that $M_n$ is homeomorphic 
to the result of Dehn filling along some of the components of $\bd X_j$. 
Let $H$ be the maximum of the $h(X_j)$. Choose an $n$ such that $h(M_n)>H$. 
By the following lemma, which is stated in greater generality than needed here, 
we have that $h(M_n)\leq H$.\end{proof}

\begin{lem} Let $Q$ and $Q^*$ be compact, connected, orientable, irreducible, 
\birr\ 3-manifolds. Suppose $Q$ is obtained by Dehn filling along some of 
the boundary components of $Q^*$. Then $h(Q)\leq h(Q^*)$. 
\end{lem}

\begin{proof} Let $V=V_1\cup\cdots\cup V_p$ be the union of the solid tori 
attached to $Q^*$ in order to get $Q$. Let $S=S_1\cup\cdots\cup S_{h(Q)}$ be 
a Haken system for $Q$. Isotop $S$ so that it meets $V$ in a collection of 
meridinal disks and the number of such disks is minimal. Let $S^*=S\cap Q^*$. 
Then $S^*$ is properly embedded in $Q^*$ and has $h(Q)$ components. It suffices 
to show that they are \inc, \binc, and pairwise non-parallel in $Q^*$. \\ 

\textit{$S^*$ is \inc\ in $Q^*$:} Suppose $D^*$ is a compressing disk for $S^*$ 
in $Q^*$. Then $\bd D^*=\bd D$ for a disk $D$ in $S$, and $D\cup D^*$ 
bounds 
a 3-ball $B$ in $Q$. Isotoping $D$ across $B$ and off $D^*$ reduces the 
number of components of $S\cap V$, contradicting minimality. \\

\textit{$S^*$ is \binc\ in $Q^*$:} Suppose $\Delta$ is a $\bd$-compressing 
bigon 
for $S^*$ in $Q^*$. Then $\bd\Delta=\alpha\cup\beta$, where $\alpha$ is a 
properly embedded  arc in $S^*$ and $\beta$ is a properly embedded spanning arc 
in an annulus $A$ in  $\bd V$ such that $\bd A=\bd D_0\cup\bd D_1$, where 
$D_0$ and $D_1$ are components of $S\cap V$. Let $E$ be the 3-ball in $V$ 
bounded by $A\cup D_0\cup D_1$. A regular neighborhood of $\Delta\cup E$ 
in $Q$ is a 3-ball across which one can isotop $S$ to remove $D_0$ and $D_1$ 
from $S\cap V$, again contradicting minimality. (Alternatively, one can show 
from this configuration that $S^*$ is compressible in $Q^*$.) \\

\textit{The components of $S^*$ are pairwise non-parallel in $Q^*$:} 
Suppose components $S^*_0$ and $S^*_1$ are parallel in $Q^*$. 
These surfaces are the intersections with $Q^*$ of components 
$S_0$ and $S_1$ of $S$. There 
is an embedding of $W^*=S^*_0\times [0,1]$ in $Q^*$ with 
$S^*_0=S^*_0\times\{0\}$, 
$S^*_1=S^*_0\times\{0\}$, and $\bd S^*_0\times[0,1]$ contained in $\bd Q^*$. 
Each component $A$ of $W^*\cap \bd V$ is an annulus for which there exists 
a 3-ball $B$ in $V$ with $A=B\cap W^*$ such that the closure of $\bd B-A$ 
consists of components of $S\cap V$. These $B$ allow one to extend the 
product structure on $W^*$ to a product stucture $W=S_0\times[0,1]$ in $Q$ 
with $S_0\times\{0\}=S_0$, $S_0\times\{1\}=S_1$, and $\bd S_0\times[0,1]$ 
contained in $\bd Q$, contradicting the fact that $S$ is a Haken system 
in $Q$. \end{proof} 

\begin{cor}  Let $Y_n$ be a sequence of compact, connected, orientable 
3-manifolds such that the complement $U_n$ of the torus boundary components of 
$Y_n$ is a finite volume hyperbolic manifold with totally geodesic boundary. \\ 
\begin{center}
If ${\ds\lim_{n\rightarrow\infty}h(Y_n)=\infty}$, then 
${\ds\lim_{n\rightarrow\infty}V\!ol(U_n)=\infty}$. 
\end{center}
\end{cor}

\begin{proof} Let $M_n$ be the double of $Y_n$ along the union $F$ of 
its non-torus boundary components. Suppose $S$ is a Haken system in $Y_n$. 
Let $\widehat{S}$ be the double of $S$ along $S\cap F$. The incompressibility 
of $\widehat{S}$ in $M_n$ follows from the incompressibility of $F$ and 
$S$ in $Y_n$ and the $\bd$-incompressibility of $S$ in $Y_n$. The 
fact that $\bd M_n$ consists of tori then implies that $\widehat{S}$ is 
$\bd$-incompressible in $M_n$. Suppose two components $\widehat{S}_0$ and 
$\widehat{S}_1$ of $\widehat{S}$ are parallel in $M_n$ via a product 
$\widehat{S}_0\times[0,1]$. This product is invariant under the involution 
which interchanges the two copies of $Y_n$. By \cite[Theorem A]{Kim-Tollefson}  
the restriction of the involution is equivalent to a product involution. 
Hence the fixed point set consists of annuli. By \cite[Lemma 3.4]{Waldhausen} 
they 
are isotopic 
to product annuli. It follows that $S_0$ and $S_1$ are parallel in $Y_n$. 
Thus $h(M_n)\geq h(Y_n)$. The result then follows from the Theorem and the 
fact that the interior $N_n$ of $M_n$ is a complete, finite volume 
hyperbolic 3-manifold with $V\!ol(N_n)=2V\!ol(U_n)$. \end{proof}

\end{document}